\documentclass[rmcite,12pt]{amsart}
\usepackage{graphicx}
\usepackage{mathrsfs}

\usepackage{amsthm,cite}

\newtheorem{theorem}{Theorem}[section]

\newtheorem{definition}{Definition}[section]
\theoremstyle{remark}
\newtheorem{remark}{Remark}[section]

\numberwithin{equation}{section}

\begin{document}

\title{Initial Coefficients of  Bi-univalent Functions}

\author[S. K. Lee]{See Keong Lee }

\address{School of Mathematical Sciences,
Universiti Sains Malaysia, 11800 USM, Penang, Malaysia}
\email{ sklee@cs.usm.my }

\author[V. Ravichandran]{V. Ravichandran}

\address{School of Mathematical Sciences,
Universiti Sains Malaysia, 11800 USM, Penang, Malaysia and
  Department of Mathematics, University of Delhi,
Delhi--110 007, India}
\email{ vravi@maths.du.ac.in}

\author[S. Supramaniam]{Shamani Supramaniam}

\address{ School of Mathematical Sciences,
Universiti Sains Malaysia, 11800 USM, Penang, Malaysia}
\email{sham105@hotmail.com}

\begin{abstract} An analytic function $f$  defined on the open unit disk $\mathbb{D}=\{z:|z|<1\}$ is bi-univalent if the function $f$ and its inverse $f^{-1}$ are univalent in $\mathbb{D}$. Estimates for the initial coefficients of bi-univalent functions $f$ are  investigated when $f$ and $f^{-1}$ respectively belong to some subclasses of univalent functions. Some earlier results are shown to be special cases of our results.
\end{abstract}



\maketitle

\section{Introduction} Let $\mathcal{S}$ be the class of all univalent analytic
functions $f$ in the open unit disk $\mathbb{D} : = \{ z \in
\mathbb{C} : |z| < 1 \}$ and normalized by the conditions $f(0)= 0$
and $f'(0)= 1$.  For $f\in\mathcal{S}$, it is well known that the $n$th coefficient is
bounded by $n$. The bounds for the coefficients give  information
about the geometric properties of these functions. Indeed, the
bound for the second coefficient of functions in the class $\mathcal{S}$
gives rise to growth, distortion, covering theorems for univalent
functions. In view of the influence of the second coefficient in the geometric properties of univalent
functions, it is important to know the bounds for the (initial) coefficients of functions
belonging to various subclasses of univalent functions. In this paper, we investigate this coefficient problem for certain subclasses of bi-univalent functions.

Recall that the Koebe one-quarter theorem \cite{duren}  ensures
that the image of $\mathbb{D}$ under every univalent function
$f\in\mathcal{S}$ contains a disk of radius 1/4. Thus every
univalent function $f$ has an inverse $f^{- 1}$ satisfying $f^{-
1}(f(z))= z$, $(z \in \mathbb{D})$ and
\[ f(f^{- 1}(w))= w, \quad  \left(|w| < r_0(f), r_0(f)
   \geq 1 / 4 \right). \]
A function $f \in \mathcal{S}$ is   \emph{bi-univalent} in $\mathbb{D}$ if both $f $ and $f^{- 1}$ are univalent in $\mathbb{D}$. Let $\sigma$
denote the class of bi-univalent functions defined in the unit disk $\mathbb{D}$.
Lewin {\cite{lewin}} investigated this class $\sigma$ and obtained the bound for the second coefficient of the bi-univalent functions. Several authors subsequently studied similar problems in this direction (see {\cite{brankir,netan69}}). A function $f\in \sigma$ is bi-starlike or strongly bi-starlike or bi-convex of order $\alpha$ if $f$ and $f^{-1}$ are both starlike, strongly starlike or convex of order $\alpha$, respectively. Brannan and Taha {\cite{brantaha}} obtained estimates for the initial coefficients of bi-starlike,strongly bi-starlike and bi-convex functions. Bounds for the initial coefficients of several classes of functions were also investigated in
\cite{rma1,rma2,pon1,pon2,fra,mish,sprasad,vravi,sri,sriv,xu1,xu2}.


An analytic function $f$ is\emph{ subordinate} to an analytic function $g$, written $f(z)\prec g(z)$, if  there is an analytic function $w:\mathbb{D}\rightarrow \mathbb{D}$ with $w(0)= 0$ satisfying $f(z)= g(w(z))$. Ma and Minda {\cite{mamin}} unified various subclasses of
starlike ($\mathcal{S^*}$) and convex functions ($\mathcal{C}$) by requiring that either of the quantity $zf'(z)/ f(z)$ or $1 + zf''(z)/ f'(z)$ is subordinate to a more general
superordinate function  $\varphi$ with positive real part in the unit disk $\mathbb{D}$,
$\varphi(0)= 1$, $\varphi'(0)> 0$,   $\varphi$ maps $\mathbb{D}$ onto a region starlike with respect to $1$ and symmetric with respect to the
real axis. The class $\mathcal{S^*}(\varphi)$ of Ma-Minda starlike functions with respect to $\varphi$ consists of functions $f \in \mathcal{S}$ satisfying the subordination $zf'(z)/ f(z)\prec
\varphi( z)$. Similarly, the class $\mathcal{C}(\varphi)$ of Ma-Minda convex functions consists of functions $f \in \mathcal{S}$ satisfying the subordination $1 +
zf''(z)/ f'(z)\prec \varphi(z)$. Ma and Minda investigated growth and distortion properties of functions in $\mathcal{S^*}(\varphi)$ and $\mathcal{C}(\varphi)$ as well as  Fekete-Szeg\"{o} inequalities  for $\mathcal{S^*}(\varphi)$ and $\mathcal{C}(\varphi)$. Their proof of  Fekete-Szeg\"{o} inequalities  requires the univalence of $\varphi$. Ali {\it et al.} \cite{rma2} have investigated Fekete-Szeg\"{o}\ problems for various other classes and their proof does not require the univalence or starlikeness of $\varphi$. In particular, their results are valid even if one just assume the function $\varphi$ to have a series expansion of the form $\varphi(z)= 1 + B_1 z + B_2 z^2 +  \cdots$, $B_1 >0$. So in this paper, we assume that $\varphi$ has series expansion $\varphi(z)= 1 + B_1 z + B_2 z^2 +  \cdots$, $B_1$, $B_2$ are real and $B_1 >0$.
A function $f$ is Ma-Minda bi-starlike or   Ma-Minda bi-convex  if both $f$ and $f^{- 1}$ are
respectively Ma-Minda starlike or convex. Motivated by the Fekete-Szeg\"o problem for the classes of  Ma-Minda starlike and Ma-Minda convex functions \cite{mamin},  Ali \emph{et al.}\ {\cite{rmabi}} recently obtained estimates of the initial coefficients for bi-univalent Ma-Minda starlike and Ma-Minda convex functions.

The present work is motivated  by the results of K\c edzierawski \cite{ked} who considered functions $f$ belonging to certain subclasses of univalent functions while its inverse $f^{-1}$ belongs to some other subclasses of univalent functions. Among other results, he obtained  the following coefficient estimates.

\begin{theorem}{\rm\cite{ked}} \label{ked1.1} Let $f\in\sigma$ with Taylor series $f(z)=z+a_2z^2+\cdots$ and $g=f^{-1}$. Then
\[ |a_2|\leq
\begin{cases}
 1.5894 & \text{ if $f\in\mathcal{S}$,    $g\in\mathcal{S}$},\\
    2      & \text{ if $f\in\mathcal{S^*}$, $g\in\mathcal{S^*}$} ,\\
  1.507 & \text{ if $f\in\mathcal{S^*}$, $g\in\mathcal{S}$},\\
  1.224 & \text{ if $f\in\mathcal{C}$,  $g\in\mathcal{S}$}.
  \end{cases}
  \]
\end{theorem}

We need the following classes investigated in \cite{rmabi,rma1,rma2}.
\begin{definition}
Let $\varphi:\mathbb{D}\rightarrow \mathbb{C}$ be analytic and $\varphi(z)=1+B_1z+B_2z^2+\cdots$ with $B_1>0$ and $B_2\in\mathbb{R}$.
 For $\alpha \geq 0$, let
\begin{align*}
\mathcal{M}(\alpha, \varphi) & :=\left\{ f\in \mathcal{S} \,:\,
(1-\alpha)\frac{zf'(z)}{f(z)}+\alpha\left(
1+\frac{zf''(z)}{f'(z)}\right) \prec \varphi(z)  \right\},\\
\mathcal{L}(\alpha,  \varphi) &: =\left\{ f\in \mathcal{S} \,:\,
\left(\frac{zf'(z)}{f(z)}\right)^\alpha \left(
1+\frac{zf''(z)}{f'(z)}\right)^{1-\alpha}\prec \varphi(z) \right\},\\
\mathcal{P}(\alpha,  \varphi) &: =\left\{ f\in \mathcal{S} \,:\,
\frac{zf'(z)}{f(z)} + \alpha \frac{z^2f''(z)}{f(z)} \prec
\varphi(z)  \right\}.
\end{align*}\end{definition}

In this paper, we obtain the estimates for the second and third coefficients of functions $f$ when
\begin{enumerate}
\item  $f\in \mathcal{P}(\alpha,  \varphi)$ and $g:=f^{-1}\in \mathcal{P}(\beta,  \psi)$, or
 $g\in \mathcal{M}(\beta,  \psi)$, or $g\in \mathcal{L}(\beta,  \psi)$,
\item $f\in \mathcal{M}(\alpha,  \varphi)$ and $g\in \mathcal{M}(\beta,  \psi)$, or
 $g\in \mathcal{L}(\beta,  \psi)$,
\item $f\in \mathcal{L}(\alpha,  \varphi)$ and $g\in \mathcal{L}(\beta,  \psi)$.
\end{enumerate}

\section{Coefficient Estimates}
In the sequel, it is assumed that $\varphi$ and $\psi$ are analytic functions  of the form
\begin{equation}
  \label{varphi1} \varphi(z)= 1 + B_1 z + B_2 z^2 + B_3 z^3 + \cdots , \quad (B_1>0)
\end{equation}and
\begin{equation}
  \label{psi1} \psi(z)= 1 + D_1 z + D_2 z^2 + D_3 z^3 + \cdots , \quad (D_1>0).
\end{equation}

\begin{theorem}\label{the1} Let $f\in\sigma$ and $g=f^{-1}$. If
  $f\in \mathcal{P}(\alpha,  \varphi)$, $g\in \mathcal{P}(\beta,  \psi)$ and $f$ of the form
  \begin{equation}
  \label{eqf} f(z)= z + \sum_{n = 2}^{\infty} a_n z^n,
\end{equation}
 then
  \begin{equation}
    \label{res1} |a_2 | \leq \frac{B_1 D_1\sqrt{B_1(1+3\beta)+D_1 (1+3\alpha)}}{\sqrt{|\sigma B_1^2 D_1^2- (1+2\alpha)^2 (1+3\beta)(B_2 -B_1)D_1^2-(1+2\beta)^2(1+3\alpha)(D_2 -D_1)B_1^2|}}\end{equation}
    and \begin{equation}
    \label{res1i} 2\sigma|a_3 | \leq  B_1(3+10\beta)+ D_1(1+2\alpha)+ (3+10\beta)|B_2-B_1|+ \frac{(1+2\beta)^2B_1^2|D_2-D_1|}{D_1^2(1+2\alpha)}
  \end{equation} where $\sigma:=2+7\alpha+7\beta+24\alpha\beta$.
\end{theorem}

\begin{proof}
 Since $f\in \mathcal{P}(\alpha,  \varphi)$ and $g\in \mathcal{P}(\beta,  \psi)$, $g=f^{-1}$. There exist analytic functions
  $u, v : \mathbb{D} \rightarrow \mathbb{D}$, with $u(0)= v(0)= 0$,
  satisfying
  \begin{equation}
    \label{newth1} \frac{zf'(z)}{f(z)}+\frac{\alpha z^{2}f''(z)}{f(z)} = \varphi(u(z))\quad
  \text{and} \quad  \frac{wg'(w)}{g(w)} + \frac{\beta w^{2} g''(w)}{g(w)} = \psi(v(w)).
  \end{equation}
  Define the functions $p_1$ and $p_2$ by
  \begin{equation*}
    \label{eqpu} p_1(z): = \frac{1 + u(z)}{1 - u(z)}=1 + c_1 z + c_2 z^2 +
     \cdots \quad  \text{and} \quad
  p_2(z): = \frac{1 + v(z )}{1 - v(z)}= 1+ b_1 z + b_2 z^2 + \cdots, \end{equation*} or, equivalently,
  \begin{equation}
    \label{equ} u(z)= \frac{p_1(z)- 1}{p_1(z)+ 1}=\frac{1}{2} \left( c_1 z
     + \left(c_2 - \frac{c_1^2}{2} \right)z^2 + \cdots \right)\end{equation} and
     \begin{equation}\label{eqv} v(z)= \frac{p_2(z)- 1}{p_2(z)+ 1}=\frac{1}{2} \left( b_1
     z     + \left(b_2 - \frac{b_1^2}{2} \right)z^2 + \cdots \right) .\end{equation}
Then $p_1$ and $p_2$ are analytic in $\mathbb{D}$ with $p_1(0)=1=p_2(0)$. Since $u, v : \mathbb{D} \rightarrow \mathbb{D}$, the functions $p_1$
and $p_2$ have positive real part in $\mathbb{D}$, and $|b_i|\leq 2 $ and $|c_i|\leq 2$.
  In view of {\eqref{newth1}},  \eqref{equ} and \eqref{eqv}, clearly
  \begin{equation}
    \label{eq2} \frac{zf'(z)}{f(z)}+\frac{\alpha z^{2}f''(z)}{f(z)}= \varphi \left(\frac{p_1(z)- 1}{p_1(z)+ 1}
    \right)\quad  \text{and} \quad  \frac{wg'(w)}{g(w)} + \frac{\beta w^{2} g''(w)}{g(w)}= \psi \left(
    \frac{p_2(w)- 1}{p_2(w)+ 1} \right).
  \end{equation}
Using \eqref{equ} and \eqref{eqv} together with {\eqref{varphi1}} and \eqref{psi1},
it is evident that
  \begin{equation}
    \label{var1} \varphi \left(\frac{p_1(z)- 1}{p_1(z)+ 1} \right)=
    1 + \frac{1}{2} B_1 c_1 z + \left( \frac{1}{2} B_1 \left(c_2 -
    \frac{c_1^2}{2} \right)+ \frac{1}{4} B_2 c_1^2 \right) z^2 + \cdots
  \end{equation}
  and
  \begin{equation}
    \label{var2} \psi \left(\frac{p_2(w)- 1}{p_2(w)+ 1} \right)=
    1 + \frac{1}{2} D_1 b_1 w + \left( \frac{1}{2} D_1 \left(b_2 -
    \frac{b_1^2}{2} \right)+ \frac{1}{4} D_2 b_1^2 \right) w^2 + \cdots .
  \end{equation}
Since $f$ has the  Maclaurin series given by
\eqref{eqf}, a computation shows that its inverse $g=f^{- 1}$ has
the expansion
\[ g(w)=f^{- 1}(w)= w -a_2 w^2 +(2 a_2^2 - a_3)w^3 + \cdots . \]
  Since
  \[ \frac{zf'(z)}{f(z)}+\frac{\alpha z^{2}f''(z)}{f(z)} = 1 + a_2(1+2\alpha) z
  +(2(1+3\alpha) a_3 -(1+2\alpha) a_2^2)z^2 + \cdots \]
  and
  \[  \frac{wg'(w)}{g(w)} + \frac{\beta w^{2} g''(w)}{g(w)}
  = 1 -(1+2\beta) a_2 w +((3+10\beta) a_2^2 - 2(1+3\beta) a_3)w^2 + \cdots,
  \]
  it follows from {\eqref{eq2}}, {\eqref{var1}} and {\eqref{var2}} that
  \begin{equation}
    \label{eq1.1n} a_2(1+2\alpha) = \frac{1}{2} B_1 c_1,
  \end{equation}
  \begin{equation}
    \label{eq1.2n} 2(1+3\alpha) a_3 -(1+2\alpha) a_2^2 = \frac{1}{2} B_1 \left(c_2 - \frac{c_1^2}{2}
    \right)+ \frac{1}{4} B_2 c_1^2,
  \end{equation}
  \begin{equation}
    \label{eq1.3n} - (1+2\beta)a_2 = \frac{1}{2} D_1 b_1
  \end{equation}
  and
  \begin{equation}
    \label{eq1.4n} (3+10\beta) a_2^2 - 2(1+3\beta) a_3 = \frac{1}{2} D_1 \left(b_2 -
    \frac{b_1^2}{2} \right)+ \frac{1}{4} D_2 b_1^2 .
  \end{equation}
   It follows from {\eqref{eq1.1n}} and {\eqref{eq1.3n}} that
  \begin{equation}
    \label{eq1.5} b_1 = -\frac{B_1 (1+2\beta)}{D_1(1+2\alpha)} c_1.
  \end{equation}
   Equations {\eqref{eq1.1n}}, {\eqref{eq1.2n}}, {\eqref{eq1.4n}} and
 {\eqref{eq1.5}} lead to
  \[ a_2^2 = \frac{B_1^2 D_1^2[B_1(1+3\beta)c_2+D_1(1+3\alpha)b_2]}{2[\sigma B_1^2 D_1^2-(1+2\alpha)^2 (1+3\beta)(B_2 -B_1)D_1^2-(1+2\beta)^2(1+3\alpha)(D_2 -D_1)B_1^2]}, \]
  which, in view of the well-known inequalities $|b_2 | \leq 2$ and
  $|c_2 | \leq 2$ for functions
  with positive real part,  gives us the desired estimate on $|a_2 |$ as asserted in
  {\eqref{res1}}.

  By using  {\eqref{eq1.2n}}, {\eqref{eq1.4n}} and {\eqref{eq1.5}}  lead to
  \begin{align*}2\sigma a_3 &= \frac{1}{2}\left[B_1(3+10\beta)c_2+D_1(1+2\alpha)b_2\right] \\&\quad+\frac{c_1^2}{4}\left[(3+10\beta)(B_2-B_1) +\frac{(1+2\beta)^2 B_1^2(D_2-D_1)}{D_1^2(1+2\alpha)}\right], \end{align*}
  and this yields the estimate given in {\eqref{res1i}}.
\end{proof}

\begin{remark} When $\alpha=\beta=0$ and $B_1=D_1=2$, then \eqref{res1} reduces to Theorem \ref{ked1.1}.
When $\beta=\alpha$ and $\psi=\varphi$, Theorem \ref{the1} reduces to \cite[Theorem 2.2]{rmabi}.
\end{remark}

\begin{theorem}\label{the2}
  Let $f\in\sigma$ and $g=f^{-1}$. If
  $f\in \mathcal{P}(\alpha,  \varphi)$ and $g\in \mathcal{M}(\beta,  \psi)$,
  then
  \begin{equation}
    \label{res2} |a_2 | \leq \frac{B_1 D_1\sqrt{B_1(1+2\beta)+D_1 (1+3\alpha)}}{\sqrt{|\sigma B_1^2 D_1^2- (1+2\alpha)^2 (1+2\beta)(B_2 -B_1)D_1^2-(1+\beta)^2(1+3\alpha)(D_2 -D_1)B_1^2|}}\end{equation}
    and \begin{equation}
    \label{res2i} {2\sigma}|a_3 | \leq  B_1(3+5\beta)+D_1(1+2\alpha)+(3+5\beta)|B_2-B_1|+\frac{(1+\beta)^2 B_1^2|D_2-D_1|}{D_1^2(1+2\alpha)}
  \end{equation} where $\sigma:=2+7\alpha+3\beta+11\alpha\beta$.
\end{theorem}

\begin{proof}
Let $f\in \mathcal{P}(\alpha,  \varphi)$ and $g\in \mathcal{M}(\beta,  \psi)$, $g=f^{-1}$. There exist analytic functions
  $u, v : \mathbb{D} \rightarrow \mathbb{D}$, with $u(0)= v(0)= 0$,
  such that
  \begin{equation}
    \label{newth2} \frac{zf'(z)}{f(z)}+\frac{\alpha z^{2}f''(z)}{f(z)} = \varphi(u(z))\quad
  \text{and} \quad  (1-\beta) \frac{wg'(w)}{g(w)} + \beta\left(1+\frac{ w g''(w)}{g'(w)} \right)= \psi(v(w)),
  \end{equation}
    Since
  \[ \frac{zf'(z)}{f(z)}+\frac{\alpha z^{2}f''(z)}{f(z)} = 1 + a_2(1+2\alpha) z
  +(2(1+3\alpha) a_3 -(1+2\alpha) a_2^2)z^2 + \cdots \]
  and
  \[(1-\beta) \frac{wg'(w)}{g(w)} + \beta\left(1+\frac{ w g''(w)}{g'(w)}\right)
   = 1 - (1+\beta) a_2 w +((3+5\beta) a_2^2 - 2(1+2\beta) a_3)w^2 +
     \cdots, \]
     then {\eqref{var1}}, {\eqref{var2}} and {\eqref{newth2}} yield
  \begin{equation}
    \label{eq2.1n} a_2(1+2\alpha) = \frac{1}{2} B_1 c_1,
  \end{equation}
  \begin{equation}
    \label{eq2.2n} 2(1+3\alpha) a_3 -(1+2\alpha) a_2^2 = \frac{1}{2} B_1 \left(c_2 - \frac{c_1^2}{2}
    \right)+ \frac{1}{4} B_2 c_1^2,
  \end{equation}
  \begin{equation}
    \label{eq2.3n} - (1+\beta)a_2 = \frac{1}{2} D_1 b_1
  \end{equation}
  and
  \begin{equation}
    \label{eq2.4n} (3+5\beta) a_2^2 - 2(1+2\beta) a_3 = \frac{1}{2} D_1 \left(b_2 -
    \frac{b_1^2}{2} \right)+ \frac{1}{4} D_2 b_1^2 .
  \end{equation}
 It follows from {\eqref{eq2.1n}} and {\eqref{eq2.3n}} that
  \begin{equation}
    \label{eq2.5} b_1 = -\frac{B_1 (1+\beta)}{D_1(1+2\alpha)} c_1.
  \end{equation}
   Equations {\eqref{eq2.1n}}, {\eqref{eq2.2n}}, {\eqref{eq2.4n}} and
 {\eqref{eq2.5}} lead to
  \[ a_2^2 = \frac{B_1^2 D_1^2[B_1(1+2\beta)c_2+D_1(1+3\alpha)b_2]}{2[\sigma B_1^2 D_1^2-(1+2\alpha)^2 (1+2\beta)(B_2 -B_1)D_1^2-(1+2\beta)^2(1+3\alpha)(D_2 -D_1)B_1^2]}, \]
  which  gives us the desired estimate on $|a_2 |$ as asserted in
  {\eqref{res2}} when  $|b_2 | \leq 2$ and
  $|c_2 | \leq 2$.

 Since  {\eqref{eq2.2n}}, {\eqref{eq2.4n}} and {\eqref{eq2.5}}  lead to
  \begin{align*}2\sigma a_3 &= \frac{1}{2}[B_1(3+5\beta)c_2+D_1(1+2\alpha)b_2] \\&\quad+\frac{c_1^2}{4} \left[(3+5\beta)(B_2-B_1) +\frac{(1+\beta)^2 B_1^2(D_2-D_1)}{D_1^2(1+2\alpha)}\right], \end{align*}
  and this yields the estimate given in {\eqref{res2i}}.
\end{proof}

\begin{theorem}\label{the3}
   Let $f\in\sigma$ and $g=f^{-1}$. If
  $f\in \mathcal{P}(\alpha,  \varphi)$ and $g\in \mathcal{L}(\beta,  \psi)$,
  then
  \begin{equation}
    \label{res3} |a_2 | \leq \frac{B_1 D_1\sqrt{2[B_1(3-2\beta)+D_1 (1+3\alpha)]}}{\sqrt{|\sigma B_1^2 D_1^2- 2(1+2\alpha)^2 (3-2\beta)(B_2 -B_1)D_1^2-2(2-\beta)^2(1+3\alpha)(D_2 -D_1)B_1^2|}}\end{equation}
    and \begin{align}
    \label{res3i}|\sigma a_3 | &\leq  \frac{1}{2}B_1(\beta^2-11\beta+16) +D_1(1+2\alpha)+\frac{1}{2}(\beta^2-11\beta+16)|B_2-B_1|\notag\\&\quad+\frac{(2-\beta)^2 B_1^2|D_2-D_1|}{D_1^2(1+2\alpha)}
  \end{align} where $\sigma:=10+36\alpha-7\beta-25\alpha\beta+\beta^2+3\alpha\beta^2$.
\end{theorem}

\begin{proof}
Let $f\in \mathcal{P}(\alpha,  \varphi)$ and $g\in \mathcal{L}(\beta,  \psi)$, $g=f^{-1}$. Then there are analytic functions
  $u, v : \mathbb{D} \rightarrow \mathbb{D}$, with $u(0)= v(0)= 0$,
  satisfying
  \begin{equation}
    \label{newth3} \frac{zf'(z)}{f(z)}+\frac{\alpha z^{2}f''(z)}{f(z)} = \varphi(u(z))\quad
  \text{and} \quad  \left(\frac{wg'(w)}{g(w)}\right)^\beta \left(1+\frac{ wg''(w)}{g'(w)}\right)^{1-\beta}= \psi(v(w)),
  \end{equation}
    Since
  \[ \frac{zf'(z)}{f(z)}+\frac{\alpha z^{2}f''(z)}{f(z)} = 1 + a_2(1+2\alpha) z
  +(2(1+3\alpha) a_3 -(1+2\alpha) a_2^2)z^2 + \cdots \]
  and
 \begin{align*}&\left(\frac{wg'(w)}{g(w)}\right)^\beta \left(1+\frac{ w
   g''(w)}{g'(w)}\right)^{1-\beta}\\& = 1 - (2-\beta) a_2 w +\Big((8(1-\beta)+\frac{1}{2}\beta(\beta+5))a_2^2 - 2(3-2\beta) a_3\Big)w^2 +
     \cdots, \end{align*}
     then {\eqref{var1}}, {\eqref{var2}} and {\eqref{newth3}} yield
  \begin{equation}
    \label{eq3.1n} a_2(1+2\alpha) = \frac{1}{2} B_1 c_1,
  \end{equation}
  \begin{equation}
    \label{eq3.2n} 2(1+3\alpha) a_3 -(1+2\alpha) a_2^2 = \frac{1}{2} B_1 \left(c_2 - \frac{c_1^2}{2}
    \right)+ \frac{1}{4} B_2 c_1^2,
  \end{equation}
  \begin{equation}
    \label{eq3.3n} - (2-\beta)a_2 = \frac{1}{2} D_1 b_1
  \end{equation}
  and
  \begin{equation}
    \label{eq3.4n} [8(1-\beta)+\frac{\beta}{2}(\beta+5)] a_2^2 - 2(3-2\beta) a_3 = \frac{1}{2} D_1 \left(b_2 -
    \frac{b_1^2}{2} \right)+ \frac{1}{4} D_2 b_1^2 .
  \end{equation}
 It follows from {\eqref{eq3.1n}} and {\eqref{eq3.3n}} that
  \begin{equation}
    \label{eq3.5} b_1 = -\frac{B_1 (2-\beta)}{D_1(1+2\alpha)} c_1.
  \end{equation}
   Equations {\eqref{eq3.1n}}, {\eqref{eq3.2n}}, {\eqref{eq3.4n}} and
 {\eqref{eq3.5}} lead to
  \[ a_2^2 = \frac{B_1^2 D_1^2[B_1(3-2\beta)c_2+D_1(1+3\alpha)b_2]}{\sigma B_1^2 D_1^2- 2(1+2\alpha)^2 (3-2\beta)(B_2 -B_1)D_1^2-2(2-\beta)^2(1+3\alpha)(D_2 -D_1)B_1^2}, \]
  which, in view of the well-known inequalities $|b_2 | \leq 2$ and
  $|c_2 | \leq 2$ for functions
  with positive real part,  gives us the desired estimate on $|a_2 |$ as asserted in
  {\eqref{res3}}.

  By using  {\eqref{eq3.2n}}, {\eqref{eq3.4n}} and {\eqref{eq3.5}}  lead to
  \begin{align*} {\sigma}a_3 &= \frac{B_1}{4}(\beta^2-11\beta+16)c_2+\frac{D_1}{2}(1+2\alpha)b_2 \\&\quad+\frac{c_1^2}{4} \left[\frac{1}{2}(\beta^2-11\beta+16)(B_2-B_1)+ \frac{(2-\beta)^2 B_1^2(D_2-D_1)}{D_1^2(1+2\alpha)}\right] \end{align*}
  and this yields the estimate given in {\eqref{res3i}}.
\end{proof}

\begin{theorem}\label{the4}
   Let $f\in\sigma$ and $g=f^{-1}$. If
  $f\in \mathcal{M}(\alpha,  \varphi)$, $g\in \mathcal{M}(\beta,  \psi)$,
  then
  \begin{equation}
    \label{res4} |a_2 | \leq \frac{B_1 D_1\sqrt{B_1(1+2\beta)+D_1 (1+2\alpha)}}{\sqrt{|\sigma B_1^2 D_1^2- (1+\alpha)^2 (1+2\beta)(B_2 -B_1)D_1^2-(1+\beta)^2(1+2\alpha)(D_2 -D_1)B_1^2|}}\end{equation}
    and \begin{equation}
    \label{res4i} {2\sigma}|a_3 | \leq  B_1(3+5\beta) +D_1(1+3\alpha)+(3+5\beta)|B_2-B_1| +\frac{(1+\beta)^2 (1+3\alpha) B_1^2|D_2-D_1|}{D_1^2(1+\alpha)^2}
  \end{equation} where $\sigma:=2+3\alpha+3\beta+4\alpha\beta$.
\end{theorem}

\begin{proof}
Let $f\in \mathcal{M}(\alpha,  \varphi)$ and $g\in \mathcal{M}(\beta,  \psi)$, $g=f^{-1}$. Then there are analytic functions
  $u, v : \mathbb{D} \rightarrow \mathbb{D}$, with $u(0)= v(0)= 0$,
  satisfying
  \begin{equation}
    \label{newth4} (1-\alpha)\frac{zf'(z)}{f(z)}+\alpha\left(1+\frac{ zf''(z)}{f'(z)}\right) = \varphi(u(z)),
   \quad  (1-\beta) \frac{wg'(w)}{g(w)} + \beta\left(1+\frac{ w g''(w)}{g'(w)}\right)= \psi(v(w)),
  \end{equation}
    Since
  \[ (1-\alpha)\frac{zf'(z)}{f(z)}+\alpha\left(1+\frac{ zf''(z)}{f'(z)}\right)
  = 1 + (1+\alpha) a_2 z +(2(1+2\alpha) a_3 - (1+3\alpha)a_2^2)z^2 +
     \cdots \]
  and
  \[(1-\beta) \frac{wg'(w)}{g(w)} + \beta\left(1+\frac{ w g''(w)}{g'(w)}\right)
   = 1 - (1+\beta) a_2 w +((3+5\beta) a_2^2 - 2(1+2\beta) a_3)w^2 +
     \cdots, \]
     then {\eqref{var1}}, {\eqref{var2}} and {\eqref{newth4}} yield
  \begin{equation}
    \label{eq4.1n} a_2(1+\alpha) = \frac{1}{2} B_1 c_1,
  \end{equation}
  \begin{equation}
    \label{eq4.2n} 2(1+2\alpha) a_3 -(1+3\alpha) a_2^2 = \frac{1}{2} B_1 \left(c_2 - \frac{c_1^2}{2}
    \right)+ \frac{1}{4} B_2 c_1^2,
  \end{equation}
  \begin{equation}
    \label{eq4.3n} - (1+\beta)a_2 = \frac{1}{2} D_1 b_1
  \end{equation}
  and
  \begin{equation}
    \label{eq4.4n} (3+5\beta) a_2^2 - 2(1+2\beta) a_3 = \frac{1}{2} D_1 \left(b_2 -
    \frac{b_1^2}{2} \right)+ \frac{1}{4} D_2 b_1^2 .
  \end{equation}
 It follows from {\eqref{eq4.1n}} and {\eqref{eq4.3n}} that
  \begin{equation}
    \label{eq4.5} b_1 = -\frac{B_1 (1+\beta)}{D_1(1+\alpha)} c_1.
  \end{equation}
   Equations {\eqref{eq4.1n}}, {\eqref{eq4.2n}}, {\eqref{eq4.4n}} and
 {\eqref{eq4.5}} lead to
  \[ a_2^2 = \frac{B_1^2 D_1^2[B_1(1+2\beta)c_2+D_1(1+2\alpha)b_2]}{2\sigma B_1^2 D_1^2- 2(1+\alpha)^2 (1+2\beta)(B_2 -B_1)D_1^2-2(1+\beta)^2(1+2\alpha)(D_2 -D_1)B_1^2}, \]
  which, in view of the well-known inequalities $|b_2 | \leq 2$ and
  $|c_2 | \leq 2$ for functions
  with positive real part,  gives us the desired estimate on $|a_2 |$ as asserted in
  {\eqref{res4}}.

  By using  {\eqref{eq4.2n}}, {\eqref{eq4.4n}} and {\eqref{eq4.5}}  lead to
   \begin{align*}{2\sigma} a_3 &= \frac{B_1}{2}(3+5\beta)c_2+\frac{D_1}{2}(1+3\alpha)b_2 +\frac{c_1^2}{4} \Big[(3+5\beta)(B_2-B_1)\\&\quad + \frac{(1+\beta)^2 (1+3\alpha) B_1^2(D_2-D_1)}{D_1^2(1+\alpha)^2}\Big]\end{align*}
  and this yields the estimate given in {\eqref{res4i}}.
\end{proof}

\begin{remark}
When $\beta=\alpha$ and $\psi=\varphi$, Theorem \ref{the4} reduces to \cite[Theorem 2.3]{rmabi}.
\end{remark}

\begin{theorem}\label{the5}
  Let $f\in\sigma$ and $g=f^{-1}$. If
  $f\in \mathcal{M}(\alpha,  \varphi)$, $g\in \mathcal{L}(\beta,  \psi)$,
  then
  \begin{equation}
    \label{res5} |a_2 | \leq \frac{B_1 D_1\sqrt{2[B_1(3-2\beta)+D_1 (1+2\alpha)]}}{\sqrt{|\sigma B_1^2 D_1^2- 2(1+\alpha)^2 (3-2\beta)(B_2 -B_1)D_1^2-2(2-\beta)^2(1+2\alpha)(D_2 -D_1)B_1^2|}}\end{equation}
    and \begin{align}
    \label{res5i} |\sigma a_3 | &\leq \frac{B_1}{2}(\beta^2-11\beta+16) +D_1(1+3\alpha)+\frac{1}{2}(\beta^2-11\beta+16)|B_2-B_1|\notag\\&\quad +\frac{(2-\beta)^2 (1+3\alpha) B_1^2|D_2-D_1|}{D_1^2(1+\alpha)^2}
  \end{align} where $\sigma:=10+14\alpha-7\beta+\beta^2+2\alpha\beta^2-10\alpha\beta$.
\end{theorem}

\begin{proof}
Let $f\in \mathcal{M}(\alpha,  \varphi)$ and $g\in \mathcal{L}(\beta,  \psi)$, $g=f^{-1}$. Then there are analytic functions
  $u, v : \mathbb{D} \rightarrow \mathbb{D}$, with $u(0)= v(0)= 0$,
  satisfying
  \begin{equation}
    \label{newth5} (1-\alpha)\frac{zf'(z)}{f(z)}+\alpha\left(1+\frac{ zf''(z)}{f'(z)}\right) = \varphi(u(z)),
\quad  \left(\frac{wg'(w)}{g(w)}\right)^\beta \left(1+\frac{ wg''(w)}{g'(w)}\right)^{1-\beta}= \psi(v(w)),
  \end{equation}
    Since
  \[ (1-\alpha)\frac{zf'(z)}{f(z)}+\alpha\left(1+\frac{ zf''(z)}{f'(z)}\right)
  = 1 + (1+\alpha) a_2 z +(2(1+2\alpha) a_3 - (1+3\alpha)a_2^2)z^2 +
     \cdots \]
  and
 \begin{align*}&\left(\frac{wg'(w)}{g(w)}\right)^\beta \left(1+\frac{ w
   g''(w)}{g'(w)}\right)^{1-\beta} \\&= 1 - (2-\beta) a_2 w +\Big((8(1-\beta)+\frac{1}{2}\beta(\beta+5))a_2^2 - 2(3-2\beta) a_3\Big)w^2 +
     \cdots, \end{align*}
     then {\eqref{var1}}, {\eqref{var2}} and {\eqref{newth5}} yield
  \begin{equation}
    \label{eq5.1n} a_2(1+\alpha) = \frac{1}{2} B_1 c_1,
  \end{equation}
  \begin{equation}
    \label{eq5.2n} 2(1+2\alpha) a_3 -(1+3\alpha) a_2^2 = \frac{1}{2} B_1 \left(c_2 - \frac{c_1^2}{2}
    \right)+ \frac{1}{4} B_2 c_1^2,
  \end{equation}
  \begin{equation}
    \label{eq5.3n} - (2-\beta)a_2 = \frac{1}{2} D_1 b_1
  \end{equation}
  and
  \begin{equation}
    \label{eq5.4n} [8(1-\beta)+\frac{\beta}{2}(\beta+5)] a_2^2 - 2(3-2\beta) a_3 = \frac{1}{2} D_1 \left(b_2 -
    \frac{b_1^2}{2} \right)+ \frac{1}{4} D_2 b_1^2 .
  \end{equation}
 It follows from {\eqref{eq5.1n}} and {\eqref{eq5.3n}} that
  \begin{equation}
    \label{eq5.5} b_1 = -\frac{B_1 (2-\beta)}{D_1(1+\alpha)} c_1.
  \end{equation}
   Equations {\eqref{eq5.1n}}, {\eqref{eq5.2n}}, {\eqref{eq5.4n}} and
 {\eqref{eq5.5}} lead to
  \[ a_2^2 = \frac{B_1^2 D_1^2[B_1(3-2\beta)c_2+D_1(1+2\alpha)b_2]}{\sigma B_1^2 D_1^2- 2(1+\alpha)^2 (3-2\beta)(B_2 -B_1)D_1^2-2(2-\beta)^2(1+2\alpha)(D_2 -D_1)B_1^2}, \]
  which, in view of the well-known inequalities $|b_2 | \leq 2$ and
  $|c_2 | \leq 2$ for functions
  with positive real part,  gives us the desired estimate on $|a_2 |$ as asserted in
  {\eqref{res5}}.

  By using  {\eqref{eq5.2n}}, {\eqref{eq5.4n}} and {\eqref{eq5.5}}  lead to
  \begin{align*} {\sigma} a_3 &= \frac{B_1}{4}(\beta^2-11\beta+16)c_2+\frac{D_1}{2}(1+3\alpha)b_2 +\frac{c_1^2}{4}\Big [(\beta^2-11\beta+16)(B_2-B_1) \\&\quad + \frac{(2-\beta)^2 (1+3\alpha) B_1^2(D_2-D_1)}{D_1^2(1+\alpha)^2}\Big]\end{align*}
  and this yields the estimate given in {\eqref{res5i}}.
\end{proof}

\begin{theorem}\label{the6}
   Let $f\in\sigma$ and $g=f^{-1}$. If
  $f\in \mathcal{L}(\alpha,  \varphi)$, $g\in \mathcal{L}(\beta,  \psi)$,
  then
  \begin{equation}
    \label{res6} |a_2 | \leq \frac{B_1 D_1\sqrt{2[B_1(3-2\beta)+D_1 (3-2\alpha)]}}{\sqrt{|\sigma B_1^2 D_1^2- 2(2-\alpha)^2 (3-2\beta)(B_2 -B_1)D_1^2-2(2-\beta)^2(3-2\alpha)(D_2 -D_1)B_1^2|}}\end{equation}
    and \begin{align}
    \label{res6i} 2|\sigma a_3 | &\leq B_1(\beta^2-11\beta+16) +D_1(8-5\alpha-\alpha^2)+(\beta^2-11\beta+16)|B_2-B_1|\notag \\&\quad+\frac{(2-\beta)^2 (\alpha^2+5\alpha-8) B_1^2|D_2-D_1|}{D_1^2(2-\alpha)^2}
  \end{align} where $\sigma:=24+3\alpha^2+3\beta^2-17\alpha-17\beta-2\beta\alpha^2-2\alpha\beta^2-12\alpha\beta$.
\end{theorem}

\begin{proof}
Let $f\in \mathcal{L}(\alpha,  \varphi)$ and $g\in \mathcal{L}(\beta,  \psi)$, $g=f^{-1}$. Then there are analytic functions
  $u, v : \mathbb{D} \rightarrow \mathbb{D}$, with $u(0)= v(0)= 0$,
  satisfying
  \begin{equation}
    \label{newth6} \left(\frac{zf'(z)}{f(z)}\right)^\alpha \left(1+\frac{ zf''(z)}{f'(z)}\right)^{1-\alpha} = \varphi(u(z)), \quad  \left(\frac{wg'(w)}{g(w)}\right)^\beta \left(1+\frac{ w
   g''(w)}{g'(w)}\right)^{1-\beta}= \psi(v(w)),
  \end{equation}
    Since
  \begin{align*} &\left(\frac{zf'(z)}{f(z)}\right)^\alpha \left(1+\frac{ zf''(z)}{f'(z)}\right)^{1-\alpha} \\&= 1 + (2-\alpha) a_2 z +\left(2(3-2\alpha) a_3
  +  \frac{(\alpha-2)^2 -3(4-3\alpha)}{2}a_2^2\right)z^2 +
     \cdots \end{align*}
  and
 \begin{align*}&\left(\frac{wg'(w)}{g(w)}\right)^\beta \left(1+\frac{ w
   g''(w)}{g'(w)}\right)^{1-\beta} \\&= 1 - (2-\beta) a_2 w +\Big((8(1-\beta)+\frac{1}{2}\beta(\beta+5))a_2^2 - 2(3-2\beta) a_3\Big)w^2 +
     \cdots, \end{align*}
     then {\eqref{var1}}, {\eqref{var2}} and {\eqref{newth6}} yield
  \begin{equation}
    \label{eq6.1n} a_2(2-\alpha) = \frac{1}{2} B_1 c_1,
  \end{equation}
  \begin{equation}
    \label{eq6.2n} 2(3-2\alpha) a_3+\frac{1}{2}[(\alpha-2)^2-3(4-3\alpha)] a_2^2 = \frac{1}{2} B_1 \left(c_2 - \frac{c_1^2}{2}
    \right)+ \frac{1}{4} B_2 c_1^2,
  \end{equation}
  \begin{equation}
    \label{eq6.3n} - (2-\beta)a_2 = \frac{1}{2} D_1 b_1
  \end{equation}
  and
  \begin{equation}
    \label{eq6.4n} [8(1-\beta)+\frac{\beta}{2}(\beta+5)] a_2^2 - 2(3-2\beta) a_3 = \frac{1}{2} D_1 \left(b_2 -
    \frac{b_1^2}{2} \right)+ \frac{1}{4} D_2 b_1^2 .
  \end{equation}
 It follows from {\eqref{eq6.1n}} and {\eqref{eq6.3n}} that
  \begin{equation}
    \label{eq6.5} b_1 = -\frac{B_1 (2-\beta)}{D_1(2-\alpha)} c_1.
  \end{equation}
   Equations {\eqref{eq6.1n}}, {\eqref{eq6.2n}}, {\eqref{eq6.4n}} and
 {\eqref{eq6.5}} lead to
  \[ a_2^2 = \frac{B_1^2 D_1^2[B_1(3-2\beta)c_2+D_1(3-2\alpha)b_2]}{\sigma B_1^2 D_1^2- 2(2-\alpha)^2 (3-2\beta)(B_2 -B_1)D_1^2-2(2-\beta)^2(3-2\alpha)(D_2 -D_1)B_1^2}, \]
  which, in view of the well-known inequalities $|b_2 | \leq 2$ and
  $|c_2 | \leq 2$ for functions
  with positive real part,  gives us the desired estimate on $|a_2 |$ as asserted in
  {\eqref{res6}}.

  By using  {\eqref{eq6.2n}}, {\eqref{eq6.4n}} and {\eqref{eq6.5}}  lead to
  \begin{align*} {2\sigma}a_3 &= \frac{B_1}{2}(\beta^2-11\beta+16)c_2 +\frac{D_1}{2}(8-5\alpha-\alpha^2)b_2+\frac{c_1^2}{4}\Big[(\beta^2-11\beta+16)(B_2-B_1)\\&\quad +\frac{(2-\beta)^2 (\alpha^2+5\alpha-8) B_1^2(D_2-D_1)}{D_1^2(2-\alpha)^2}\Big]\end{align*}
  and this yields the estimate given in {\eqref{res6i}}.
\end{proof}

\begin{remark}
When $\beta=\alpha$ and $\psi=\varphi$, Theorem \ref{the6} reduces to \cite[Theorem 2.4]{rmabi}.
\end{remark}


\end{document}